\documentclass[a4paper,11pt]{article}
\usepackage[a4paper,top=1.1in,bottom=1.1in,left=1.1in,right=1.1in]{geometry}
\pdfoutput=1
\usepackage{epsfig}
\usepackage{graphicx}
\usepackage[latin1]{inputenc}
\usepackage[T1]{fontenc}
\usepackage{mathptmx}
\usepackage{url}

\usepackage{amsthm,amssymb,amsmath,latexsym} 
\usepackage[ruled,vlined]{algorithm2e}
\usepackage{hyperref}
\usepackage{bm}  

\SetKwInOut{Input}{Input}
\SetKwInOut{Output}{Output}
\newcommand{\CC}{\ensuremath{\mathbb{C}}\xspace}

\newcommand{\PP}{\ensuremath{\mathbb{P}}\xspace}
\newcommand{\QQ}{\ensuremath{\mathbb{Q}}\xspace}
\newcommand{\RR}{\ensuremath{\mathbb{R}}\xspace}

\theoremstyle{definition}
\newtheorem{remark}{Remark}

\DeclareMathOperator{\PGL}{PGL}
\DeclareMathOperator{\codim}{codim}
\DeclareMathOperator{\Rank}{Rank}

\newtheorem{theorem}{Theorem}
\newtheorem{prop}[theorem]{Proposition}
\newtheorem{corollary}[theorem]{Corollary}
\newtheorem{lemma}[theorem]{Lemma}
\newtheorem{example}{Example}
\newtheorem{definition}{Definition}

\begin{document}

\title{Matrix Representations by Means of Interpolation}

\author{Ioannis Z.~Emiris$^{1,2}$, Christos Konaxis$^{1,2}$,
        Ilias S.~Kotsireas$^{1,3}$,\\ and Cl\'ement Laroche$^{1,2}$
        \\[1em]
        $^1$Department of Informatics and Telecommunications,\\
	National Kapodistrian University of Athens, Greece\\ 
       $^2$ATHENA Research Innovation Center, Maroussi, Greece\\
       $^3$Department of Physics and Computer Science,\\
       Wilfrid Laurier University, Waterloo, Canada
        }

\maketitle

\begin{abstract}
We examine implicit representations of parametric or point cloud models, based on interpolation matrices,
which are not sensitive to base points.
We show how interpolation matrices can be used for ray shooting 
of a parametric ray with a surface patch, including the case of high-multiplicity intersections. 
Most matrix operations are executed during pre-processing since they solely depend on the surface.
For a given ray, the bottleneck is equation solving. Our Maple code handles 
bicubic patches in $\le 1$~sec, though numerical issues might arise.
Our second contribution is to extend the method to parametric space curves and, generally,
to codimension $> 1$, by computing the equations of (hyper)surfaces intersecting precisely at the given object. 
By means of Chow forms, we propose a new, practical, randomized
algorithm that always produces correct output but possibly with a non-minimal number of surfaces. 
For space curves, we obtain 3 surfaces whose polynomials are of near-optimal degree; in this case, computation 
reduces to a Sylvester resultant.
We illustrate our algorithm through a series of examples and compare our Maple prototype with other methods implemented in Maple
i.e., Groebner basis and implicit matrix representations.
Our Maple prototype is not faster but yields fewer equations and seems more robust than Maple's {\tt implicitize};
it is also comparable with the other methods for degrees up to 6.

\bigskip\noindent
\textit{Key\;words:}\/
Matrix representation;\, sparse resultant;\, Chow form;\, ray shooting;\, randomized algorithm;\, implicitization;\, space curve;\, Maple implementation
\end{abstract}

%
%

\section{Introduction}\label{Sintro}

In manipulating geometric objects, it is essential to possess robust
algorithms for changing representation. This paper considers three
fundamental representations, namely implicit, 
parametric (possibly with base points), and point clouds.
We offer two  
results on algorithms operating on the latter two representations. 
Our algorithms construct implicit equations or implicit matrix representations of 
parametric curves and (hyper)surfaces by means of
some type of interpolation through points.
In general, approaches to implicitization include resultants,
Gr\"obner bases, moving lines and surfaces, and interpolation techniques.

Implicit matrix representations are quite robust, since they do
not require developing the implicit equation; instead, they
reduce geometric operations on the object to matrix algebra, e.g.\
\cite{EmKaKoGmod,BusLBa12}.  
Our matrix representation uses interpolation matrices.
The method has been developed for plane curves, and (hyper)surfaces.
The columns of the matrix are indexed by a superset
of the implicit monomials, i.e., those in
the implicit polynomial with non-zero coefficient \cite{EmKaKoLBspm}.
This monomial set is determined quite tightly for parametric models,
by means of the sparse resultant of the parametric polynomials,
thus exploiting the sparseness of the parametric and implicit polynomials.
When the object is given as a point cloud, this superset is taken to be
all monomials with total degree $\le \delta$, where $\delta$ 
is an educated guess of the total degree of the implicit polynomial.

Our contribution is twofold.
First, we reduce certain geometric operations, most notably ray shooting,
to univariate solving, after preprocessing by means of matrix operations.
To check whether the computed intersections of the ray with the implicit surface lie
on the given parametric patch, we employ inversion by means of bivariate solving.
Our method is not sensitive to intersections of high multiplicity, nor to surfaces with base points,
and it can be applied to objects given either parametrically or as a point cloud.
But it exhibits numerical issues when the interpolation matrix dimension is large,
say of a few hundreds.
On the other hand, the straightforward method of solving the system 
by equating the parametric expressions to the ray coordinates typically
exhibits robustness issues, especially at the presence of base points.
We present experiments based on an open-source, Maple implementation.

Second, we extend the method to varieties of codimension higher than~1.
Given a variety in parametric form, and starting with the powerful and classic theory of Chow forms,
which generalizes resultant theory, we design an original, practical
algorithm that uses randomization in order to
compute implicit hypersurfaces containing the variety.
The question is reduced to computing a (sparse) resultant.
For space curves in 3D-space, our method yields
a small number of implicit polynomials defining the curve set-theoretically.
Their degree is near to the optimal degree and
the algorithm is significantly simpler and is expected to be
much faster than computing the optimal-degree polynomials.
Self-intersections do not affect the degree of the output.
\begin{figure}[ht]
\includegraphics[scale=0.55]{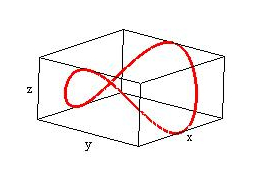}
\includegraphics[scale=0.28]{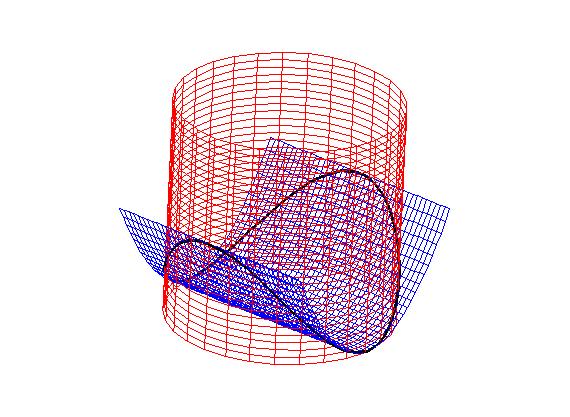}
\centering\caption{Left: The space curve of Exam.~\ref{Exam:2cylinders}.
Right: The two implicit surfaces defining the curve.\label{F2cylinders}}
\end{figure}

The rest of the paper is organized as follows:
Sec.~\ref{Sprevious} overviews previous work.
Sec.~\ref{Srayshoot} examines ray shooting when the surface is
represented by interpolation matrices.
Sec.~\ref{Schow} employs the {Chow form} to handle higher codimension varieties.
We conclude with future work and open questions.


\section{Previous work}\label{Sprevious}

A direct method to reduce implicitization to linear algebra is to construct a square 
matrix $M$, indexed by all possible monomials in the implicit equation
(columns) and different values (rows) at which all monomials get evaluated.
The vector of coefficients of the implicit equation is in the kernel of $M$,
even in the presence of base points.
This idea has been extensively used
\cite{CGKW00,Dokk01,MarMar02,BajIhm89}. 
In \cite{EmKaKoLBspm,EmKaKoGmod},
sparse elimination theory is employed to predict the implicit monomials 
and build the interpolation matrix.

Resultants, and their matrix formulae, have been used to express
the implicit surface equation, e.g., in \cite{MaCa92a}, under the assumption of no base points.
This technique reduces ray shooting (and surface-surface intersection)
to substituting the parameterization of one object into the resultant matrix,
then solving the resulting univariate polynomial.

Our method, unlike previous methods, works also for objects defined by
point clouds. Moreover, the interpolation matrices are mostly numeric,
except for one symbolic row that captures the second object, namely the
ray or the second surface.
This allows for several preprocessing operations to be executed numerically.
In \cite{EmKaKoGmod}, Membership and Sidedness predicates are expressed
in terms of operations on the interpolation matrix.
This approach is extended to ray shooting here.

A modern method for representing implicit equations by matrices is based
on the theory of syzygies, e.g., \cite{SedChe95,Buse14}, and handles base points.
It yields matrices whose entries are linear polynomials in the implicit variables
and they indirectly represent implicit objects: their rank drops exactly on the curve or surface.
They allow for geometric operations, such as surface-surface intersection \cite{BusLBa12} and, more
recently, ray shooting \cite{AlBuDoSh16}, to be executed by linear algebra operations.
Their advantage is that the matrices are much smaller than interpolation
matrices, and allow for inversion by solving an eigenproblem on these matrices.

Chow forms have been studied in computational algebra, in
particular for varieties of codimension $\ge 2$, since they provide a
method to describe the variety by a single polynomial, e.g.\
\cite{GKZ,DalStu95}. 
Given a $d$-dimensional irreducible variety $V$ in $n$-dimensional space,
and $d+1$ linear forms $H_i$ on $V$, the {\em Chow form} $R_V$ of $V$ is a
polynomial whose variables are the coefficients of the linear forms and
that vanishes whenever the linear forms vanish on $V$.

For example, the Chow form of a space curve in projective 3D-space is a polynomial in the
indeterminates $u_0,u_1,u_2,u_3, v_0,v_1,v_2,v_3$ that vanishes whenever
the planes
\begin{align}\label{Eq:Planesspacecurve}
\nonumber & H_0 = u_0 x_0 + u_1 x_1 + u_2 x_2 + u_3 x_3 = 0,\\[-6pt]
&\\[-6pt]
\nonumber & H_1 = v_0 x_0 + v_1 x_1 + v_2 x_2 + v_3 x_3 = 0,
\end{align}
intersect on the curve.

There are many ways to describe $R_V$. The most common is by using the maximal minors
of the $(d+1)\times (n+1)$ matrix whose rows are the normals to the hyperplanes $H_i$, see Sec.~\ref{SSspacecurves}.
These are known as {\em Pl\"ucker coordinates} or {\em brackets} and are the variables of $R_V$.
Equivalently, the {\em dual Pl\"ucker coordinates} or {\em dual brackets} can be used;
these are the maximal minors of a $(n-d) \times (n+1)$ matrix whose rows are $n-d$ points that span the 
intersection $L$ of the hyperplanes $H_i$.
There are algorithms to recover the implicit or affine equations of $V$ from $R_V$ \cite{GKZ,Stu08}. Here we 
adopt a practical method that avoids conversion but may yield a multiple
of $R_V$. Our algorithm requires a parametric representation of $V$.


\subsection{Interpolation matrices}\label{Smatrixmethod}

This section describes implicitization 
by an interpolation matrix $M$ for a plane curve or a (hyper)surface
given either in parametric form, or, as a point cloud.
If the input object is \emph{affinely} parameterized by rational functions
\begin{equation}\label{Eaff_parametrization}
x_i=f_i(t), \,\, i=1,\ldots,n,\,\, t=(t_1,\ldots,t_{n-1})
\end{equation}
then it is possible to predict the implicit monomials.
This set is included in the \emph{predicted (implicit) polytope} computed
by software \texttt{ResPol}\footnote{\url{http://sourceforge.net/projects/respol}} \cite{EmFiKoPeJ}. 
If the input is a point cloud, we consider a coarse estimation of
the monomial set by guessing the total degree of the variety
and taking all the monomials of that degree or lower.

The set $S$ of monomials is used to construct a numerical matrix $M$,
expressing a linear system
whose unknowns are the coefficients $c_i$ of the monomials
with exponents in $S$ in the implicit polynomial, as discussed above.
To obtain the linear system in the $c_i$ we substitute $x_i$ 
by its rational parametric expression $x_i=f_i(t)$
into equation $\sum_{i=1}^{|S|} c_i {x}^{a_i}=0$, \, $x^{a_i}:=x_1^{a_{i1}}\cdots{}x_n^{a_{in}}$,
and evaluate the parameters $t$ at \emph{generic} points 
(randomized in practice) $\tau_k \in \CC^{n-1}, ~ k = 1,\dots, \mu,~ \mu\geq |S|$, 
avoiding values that make the denominators of the
parametric expressions close to~0.

Letting $m_i=m_i(t)$ denote the monomial $x^{a_i}$ after substituting $x_i$ by its parametric expression
in \eqref{Eaff_parametrization},
and $m_i |_{t=\tau_k}$ its evaluation at $t=\tau_k$, we end up with a matrix $M$ of the form:
\begin{equation}\label{Eimatrix}
 M = \begin{bmatrix}
m_1 |_{t=\tau_1} & \cdots & m_{|S|} |_{t=\tau_1}\\[3pt]
\vdots                     & \cdots &  \vdots   \\[3pt]
m_1 |_{t=\tau_{\mu}} & \cdots & m_{|S|} |_{t=\tau_{\mu}}
\end{bmatrix}.
\end{equation}
Typically $\mu=|S|$ for performing exact kernel computation, and
$\mu= 2|S|$ for approximate numeric computation.
When constructing matrix $M$ it is assumed that the parametric hypersurface
is sampled sufficiently gen\-erically by evaluating the parametric expressions at
random points $\tau_k \in \CC^{n-1}$. 

Let us recall some further properties.
Let $M'$ be the $(|S|-1)\times{}|S|$ numeric matrix obtained by
evaluating the monomials $S$ at $|S|-1$ points $\tau_k,\, k=1\ldots,|S|-1$.
We obtain the $|S|\times{}|S|$ matrix $M(x)$, which is numeric except for its last row, 
by appending the row of monomials $S$ to matrix $M'$:
\begin{equation}\label{Ematx}
M(x) = \begin{bmatrix} M' \\[2pt] {S}(x) \end{bmatrix}.
\end{equation}

Notice that matrices $M'$, $M$ and $M(q)$, for $q$ a point lying on the hypersurface 
have the same kernel of corank $r$.
Matrix $M(x)$ has an important property:

\begin{lemma}\label{Lmatrix} {\rm\cite[Lem.7]{EmKaKoGmod}}
Assuming $M'$ is of full rank, then $\det M(x)$ equals
the implicit polynomial $p(x)$ up to a constant.
\end{lemma}


\section{Ray shooting} \label{Srayshoot}

This section examines the geometric operation of
ray shooting implemented by means of linear algebra so as to avoid computation
of the implicit equation.
Our main motivation is to complete the properties of interpolation matrices
used in expressing the implicit form of objects initially given either
by a parameterization or a point cloud.
Our experimental results are preliminary but show that our approach can be
competitive to well-established methods possessing optimized implementations.

In particular, the representation by interpolation matrices
reduces ray shooting at a surface patch
to operations on the interpolation matrix and univariate equation solving.
An analogous approach applies to computing surface-surface intersections
by determining the implicit curve of the intersection in the plane of
the parameters of one surface,  discussed in Sec.~\ref{Sfuture}.

Given a parameterization as in \eqref{Eaff_parametrization},
the (hyper)surface is represented by an interpolation matrix
of the form~(\ref{Ematx}). The standard case of surfaces is $n=3$.
We substitute the parametric ray into the monomials $S(x)$ of the surface's
implicit support. This defines the last row of the interpolation
matrix.
Let the ray be given by
$$
x_i=r_i(\rho),\, i=1,\ldots,n, \; \rho>0,
$$
where the $r_i\in\RR[\rho]$ are linear polynomials in $\rho$.
The ray may lie in arbitrary dimension but again, the standard
3D case is $n=3$.
Starting with the square matrix $M$ in expression~(\ref{Ematx}),
we substitute $x_i$ by $r_i(\rho)$ in the monomials of the last row $S(x)$:
\begin{equation}
M(\rho) = \begin{bmatrix}
M' \\[2pt] S(\rho)
\end{bmatrix} ,
\end{equation}
where $S(\rho)$ is a row vector of monomials in $\rho$.
The intersection points correspond to the real roots of
polynomial equation $p(\rho)=\det M(\rho)=0$, which expresses the implicit equation
by Lem.~\ref{Lmatrix}, evaluated at $x_i=r_i(\rho)$. Of course, if the ray
is totally included in the (hyper)surface, then $p$ is identically $0$. In the following,
we assume that it is not the case.
Lastly, our algorithm uses inversion of the parameterization
to check whether the intersection points lie on the patch.

The structure of $M(\rho)$ allows us to perform numeric operations so as to obtain
another matrix with same determinant, that is easier to compute.
Our goal is to use some preprocessing time, without knowledge of the ray,
so as to accelerate the actual ray shooting computation.
This preprocessing can be considered as a better representation of the
surface, amenable to geometric operations.
We apply a PLU-decomposition or QR-decomposition on $M'$.

\subsection{Computation}

It shall be easier to discuss matrix operations on the transpose matrices,
namely $M^T$ and $(M')^T$, although mathematically equivalent with
operations on the original matrices.
Below, we employ the PLU-decomposition, but we have also
experimented with the QR-decomposition with analogous results.

First, consider the standard PLU-decomposition on $|S| \times |S|-1$
matrix $(M')^T$:
$$
(M')^T = PLU,
$$
where $P,L$ are square $|S|\times |S|$ permutation and lower-triangular
matrices, respectively, where $L$ has a unit diagonal,
and $U$ is a numeric rectangular $|S| \times |S|-1$
upper triangular matrix, with the last row equal to zero.
We use the above decomposition to derive an analogous decomposition
of $M(\rho)^T$.
Clearly, $P^{-1}=P^T$, hence we can write square $|S|\times |S|$ matrix
$L^{-1}P^{T} M(\rho)^T$ as follows:
\begin{equation}\label{Emm}
L^{-1}P^{T} \; [\, (M')^T\, | \; S(\rho)^T ] =
	[\, U\, | \; L^{-1}P^{T} S(\rho)^T ],
\end{equation}
where $[\, \cdot\, |\,\cdot\, ]$ indicates concatenation of matrices,
and $S(\rho)$ is a row vector, hence the right-hand side is clearly a
square $|S|\times |S|$ matrix.

Now, $\det M(\rho)= \det ( L^{-1}P^{T} M(\rho)^T )$ because $\det L=1$.
To compute the determinant, we use
the right-hand side of equation~(\ref{Emm}), which is a square upper
triangular matrix. Its diagonal is the diagonal of $U$ with one last
element equal to the last element of column $L^{-1}P^{T} S(\rho)^T$.
Since $U$ is numeric, the determinant equals,
up to a constant nonzero factor, the last element on the column
$L^{-1}P^{T} S(\rho)^T$.
Hence, it suffices to store the last row of $L^{-1}P^{T}$,
which depends on the surface, and
compute its inner product with vector $S(\rho)^T$, for any given ray.
This discussion establishes the following:

\begin{lemma}\label{Ldecomp}
Given $M(\rho)^T=[ (M')^T \, | \, S(\rho)]$ as defined above,
consider the PLU-decomposition $(M')^T=PLU$, which
can be computed at preprocessing,
and does not require knowledge of the particular ray.
It is possible to compute the determinant $\det M(\rho)$, up to
a nonzero scalar factor, as the inner product of the last row
of $L^{-1}P^{T}$ multiplied by column vector $S(\rho)^T$.
An analogous result follows if one computes the QR-decom\allowbreak position
$(M')^T=QR$. 
\end{lemma}

Given $p(\rho)=\det M(\rho)$, we employ any state-of-the-art real solver
to approximate all positive real roots $\rho_0$.
For every root $\rho_j >0$, for some $j$, we obtain the intersection point
$x_i= r_i(\rho_j)$ in implicit space.
An alternative is to use solvers that only require values of the
polynomial $p(\rho)$ and, eventually, its derivative $p'(\rho)$,
such as the Newton-Raphson method.
We have experimented with this method and the results are comparable
in terms of numerical accuracy and speed.

The last step consists in computing
the corresponding vector $(t_1,t_2,\dots,t_{n-1})$ $\in\RR^{n-1}$ that maps
to this implicit point, by inverting the parameterization. Hence we check
whether the point lies on the patch of interest.
Assuming a normal and proper parameterization, there is a unique solution
$(t_1,t_2,\dots,t_{n-1})$ $\in\RR^{n-1}$ such that $r_i(\rho_j)=f_i(t_1,t_2,\dots,t_{n-1})$.
Since we are dealing with patches here, we may not want to consider the missing points of 
the parameterization if we do not have a normal (surjective) parameterization. In that situation, 
the normal parameterization hypothesis can be forgotten.

There are several ways to solve the inversion problem, including
methods for solving the overconstrained system
$r_i(\rho_j)=f_i(t),\, i=1,\dots,n$ for every $\rho_j$.
Alternatively, we have used real solving methods applied to 
the following well-constrained  system of $n$ polynomials
in $t_1,\dots,t_{n-1}$: $r_i(\rho_j)-f_i(t)=0$, for $i=1,\dots,n$.
Then, we substitute all real roots into the last polynomial
$r_{n}(\rho_j)-f_{n}(t)$, and keep the root that satisfies it,
or best satisfies it if approximate numerical computation is involved.
This approach works satisfactorily for $n=3$ in our experiments.

\subsection{Asymptotic and practical complexity}

Let us first consider the asymptotic bit complexity, if we omit preprocessing.
Let $D$ be the maximum total degree of any monomial in $S$, hence
polynomial $p(\rho)$ has degree $D$.
Real solving of this polynomial costs 
$O^*(D^{3}|S|)$, e.g.\ \cite{TsiEmiTcs}. 
Inversion is then required to check which solutions belong to the relevant
patch; it is reduced to bivariate solving and evaluation of the polynomials
$r_i(\rho_j)-f_i(t)$ defined by the parametric equations,
for every real root $\rho_j$ of $p(\rho)=0$.
If these polynomials have total degree $\delta$ in $s,t$ and bitsize $\tau$,
then each system is solved in bit complexity $O^*(\delta^{12}+\delta^{10}\tau^2)$,
see e.g.\ \cite{DiEmTs},
and this is repeated for each of the $O(D)$ roots $\rho_j$.
It turns out that this is the bottleneck in practice, which can
be explained by the analysis since $D$ is only a few times larger than $d$.
For example, the bicubic surface has $D=18$ and $\delta=6$.

We now experiment with code developed on top
of the basic method~\footnote{Maple implementation at \url{http://ergawiki.di.uoa.gr/experiments/simpl.mpl}},
running Maple~14 on a laptop with a 3~GHz Intel Core~i5 processor.
Runtimes are averaged over at least~3 runs.

\begin{example}
The standard benchmark of a bicubic surface has
an implicit equation of total degree~18, with a
support of~715 monomials.
Here, the predicted support is optimal.
However, matrix size makes computations such as PLU- or QR-decompositions
over $\QQ$ infeasible. 
We use floating-point numbers, but then numerical error becomes significant. 
A concrete ray we experimented with was
$ x_1=\rho-13,\; x_2 =12+\rho,\; x_3 =3-5\rho .  $
If we omit preprocessing, the total runtime is up to $0.12$~sec,
and the bottleneck is inversion.
\end{example}

Our method may be juxtaposed to the direct approach of
solving the $3\times 3$ system $f_i(t_1,t_2)=r_i(\rho)$, $i=1,2,3$,
which takes about $0.02$~sec on Maple
and is substantially more accurate numerically, on this example.
But numerical issues arise for the direct approach when
the parameterization has base points.
Both direct and our method handle intersection points with
high multiplicity easily.

Let us also compare with the matrix representation based on syzygies
and employed for ray shooting by \cite{AlBuDoSh16}.
This method is very robust and can handle bicubic surfaces with preprocessing and ray-shooting runtimes
both below 1~msec, using a C++ implementation and state-of-the-art libraries such as LAPACK and Eigen.
The second runtime may increase in the case of multiple intersection points, but base points
pose no limitation.
We have focused on methods potentially offering exact solutions if given
enough precision; hence, we have not compared to methods relying on
a triangular mesh since they are completely different and may need
to refine candidate solutions against the parametric patch.

To avoid numerical issues, our approach may employ exact rational computation
for the PLU- or QR-decompositions. This is feasible with medium-size
inputs: e.g., for the Moebious surface, a predicted support of size 21 allows for QR 
decomposition over the rationals and solving following Lem.\ref{Ldecomp} in 2 sec. Note that this surface
has complex base points which make certain real solvers to fail for simple rays.
On simpler examples, e.g., the crossed surface $x_1 = t_1 , \, x_2 = t_2 ,\; x_3 = t_1^2 t_2^2 ,$
with implicit equation $x_3 - x_1^2x_2^2$, the direct method is up to~2 times faster and
equally accurate as ours.


\section{Beyond hypersurfaces}\label{Schow}

This section examines the case of varieties of codimension greater than~1.
A fundamental tool in algebraic elimination is the Chow form of a variety.  
\begin{definition}\label{Dchow}
Let $V\subset \PP^n$  be a $d$-dimensional irreducible variety and
$H_0,\ldots,H_{d}$ be linear forms where
\begin{equation}\label{Eq:hyperplanes}
 H_i=u_{i0}x_0+\dots+u_{in}x_n, \quad i=0,\ldots,d
\end{equation}
and $u_{ij}$ are new variables, $0\leq i \leq d,\, 0\leq j \leq n$.
The Chow form $R_V$ of $V$ is a polynomial in the variables $u_{ij}$ such that
$$
R_V(u_{ij})=0 \Leftrightarrow V\cap \{ H_0=0,\ldots,H_d=0\}\neq \emptyset.
$$
\end{definition}

The intersection of the $d+1$ hyperplanes $H_i$ defined in
equation~\eqref{Eq:hyperplanes} is a
$(n-d-1)$-dimensional linear subspace $L$ of $\PP^n$.
The set of  $k$-dimensional linear subspaces $L$ of $\PP^n$
is the Grassmannian $G(k,n)$.
Those subspaces $L$ that intersect $V$ form a hypersurface $Z(V)$
on the Grassmannian $G(n-d-1,n)$ by \cite[Prop.2.2,p.99]{GKZ}.
The Chow form is the unique polynomial, up to a scalar factor, that defines $Z(V)$.
To recover $V$ from $Z(V)$ we have the following:

\begin{prop}{\em\cite[Prop.2.5,p.102]{GKZ}}\label{Pgkz}
A $d$-dimensional irreducible subvariety $V\subset\PP^{n}$ is uniquely
determined by its Chow form. More precisely, a point $\xi\in\PP^n$ lies in
$V$ if and only if any $(n-d-1)$-dimensional plane containing $\xi$ belongs to
the Chow form.
\end{prop}

Consider a variety $V$ as in Prop.~\ref{Pgkz} parameterized as 
\[
x_i=f_i(t),\,\, i=0,\dots,n, \,\, t=(t_0:\dots :t_{d}),
\]
and $d+1$ hyperplanes $H_0=\dots=H_d=0$,
where $H_i$ is defined as in equation~\eqref{Eq:hyperplanes}. 
Substituting $x_i=f_i(t)$ in every equation $H_i=0$,
we obtain an overdetermined system of equations in the parameters $t$
which we can eliminate by using resultants.
This reduces the computation of any Chow form to the computation of
a multivariate resultant:

\begin{corollary}
Consider any $V\subset\PP^n$ of dimension~$d$, with 
parameterization $x_i=f_i(t)$, $i=0,\dots,n$.
Then, the Chow form  $R_V$ is the resultant of the hyperplane equations
$H_i$, where one eliminates $t$.
\end{corollary}

Recall that $R_V$ is a polynomial in the Pl\"ucker coordinates, i.e., the maximal
minors of the $(d+1)\times(n+1)$ matrix whose rows are the normals to the hyperplanes $H_i$.
To obtain a representation for  $V$ as intersection of implicit hypersurfaces
from its Chow form one can apply \cite[Cor.2.6,p.102]{GKZ},
or \cite[Prop.3.1]{DalStu95}. Both methods reduce to applying a rewriting algorithm
such as the one in \cite{Stu08} and typically yield more implicit polynomials than necessary,
all of degree equal to the degree of $V$.

To illustrate the approach in \cite{DalStu95},
let us focus on varieties of codimension~2 in $\PP^3$, i.e., space curves.
Consider the planes $H_0, H_1$ in equation \eqref{Eq:Planesspacecurve}.
The Chow form $R_V$ is a polynomial in the brackets $[i,j]$, where $[i,j]$ denotes
the maximal minor indexed by the columns $0\le i,j \le 3$,
of the matrix
\[
U:=\begin{bmatrix}  u_0 & u_1 & u_2 & u_3 \\[1pt] v_0 & v_1 & v_2 & v_3 \end{bmatrix}.
\]
$R_V$ is then rewritten as a polynomial in the dual brackets $[[i,j]]$.
The dual brackets are substituted by the determinant $u_iv_j-u_jv_i$ of the corresponding minor of $U$.
Finally, the result is expanded as a polynomial whose variables are polynomials in the $v_0, v_1, v_2, v_3$ and its
coefficients are polynomials in the $u_0, u_1, u_2, u_3$.
The latter polynomials are all of degree equal to the degree of $V$
and form a system of implicit equations of $V$.


\subsection{Implicit representations for space curves}\label{SSspacecurves}

We now derive implicit representations by
avoiding complex computations, such as the rewriting algorithm.
Suppose that we have a space curve parameterized as
\[ 
x_i=f_i(t),\quad i=0,\ldots,3, \quad t=(t_0:t_1). 
\]
Let line $L$ be defined by a symbolic point
$\xi=(\xi_0:\dots :\xi_3)$ and a sufficiently generic point $G\not\in V$.
Define two planes $H_0, H_1$ that intersect along $L$, by choosing one more random
point $P_0$ and $ P_1$, respectively.
Let their equations be as in \eqref{Eq:Planesspacecurve},
where the coefficients now are \emph{linear polynomials} in $\xi$.
The Sylvester resultant of this system, where we set $x_i=f_i(t)$,
eliminates $t$ and returns a polynomial in $\xi$ which vanishes on $V$ (but not only on $V$),
thus offering a necessary but not sufficient condition,
see Algorithm~\ref{Achow}.

\begin{lemma}\label{Lsurfacedegree}
Let $\delta=\max\{\deg f_i(t), \, i=0,\ldots,3\}$ and $R_V^G$ be the Sylvester resultant of 
$$
H_0(f_0(t):\ldots:f_3(t)) \mbox{ and }\, H_1(f_0(t):\ldots:f_3(t)).
$$
Then $R_V^G$ factors into a degree $\delta$ polynomial
defining a surface $S_V^{G_i}$ containing the space curve $V$, and a polynomial $E_L^\delta$,
where $E_L$ is a linear polynomial defining the plane passing through points $G, P_0, P_1$.
\end{lemma}

\begin{proof}
The degree of the Sylvester resultant in the coefficients of each of the
$H_0, H_1$, is $\delta$, and its total degree is $2\delta$;
$\xi$ is involved (linearly) in the coefficients
of both $H_0$ and $H_1$, since it is taken to lie in the intersection of the two planes.
Hence the degree of the sought polynomial in $\xi$ is $2\delta$.

It vanishes only in two cases:
if $\xi$ belongs to the plane defined by $G$, $P_0$ and $P_1$,
or if $\xi$ belongs to a line passing by $G$ and intersecting $V$.
Hence we can divide (possibly several times) the sought 
polynomial in $\xi$ by the equation $E_L$ of 
the plane defined by $G$, $P_0$ and $P_1$,
thus obtaining an equation of the conical surface of vertex $G$ and directrix $V$. 
Since such a conical surface is of degree $\delta$,
its equation is $Res(H_0, H_1) / E_L^\delta$.
\end{proof}

\begin{theorem}\label{Tthreesurf}
Let $p:\mathbb{P}^1(\mathbb{R}) \rightarrow \mathbb{P}^3(\mathbb{R})$
be a rational parameterization 
of a space curve $V$ and $S_V^{G_i},\, i=1, 2, 3$ 
be three conical surfaces
obtained by the method above with 3 different random points $G_i \notin V$. 
We distinguish two cases.
\begin{enumerate}
\item 
If $V$ is not planar and the points $G_i$ are not collinear, 
then $V = S_V^{G_1} \cap S_V^{G_2} \cap S_V^{G_3} $.
\item 
If $V$ is contained in a plane $\mathcal{P}$, at least two of  
the points $G_i$ are not in $\mathcal{P}$ and if the $G_i$
are not collinear, then $V = \cap_{i=1,2,3} S_V^{G_i}$.
\end{enumerate}
\end{theorem}

\begin{proof}
\emph{Case (1)}. 
Since $S_V^{G_i}$ are three different cones (or cylinders when $G_i$ are at infinity), 
they have no 2-dimensional component in common. 
We first reduce the problem to the case 
\begin{align*}
 G_1&=P_x:=(1:0:0:0), \quad G_2=P_y:=(0:1:0:0), \quad G_3=P_z:=(0:0:1:0).
\end{align*}
We shall prove that an algebraic space curve is the intersection of the 3 cylinders
spanned by the curve itself and of directions $\vec x_1, \vec x_2$ and $\vec x_3$ respectively.

Let $\mathcal{C}$ be an irreducible component of $(S_V^{G_1} \cap S_V^{G_2} \cap S_V^{G_3})$.
Since $G_1, G_2, G_3$ are not collinear, there exists a map $\phi \in \PGL(4,\mathbb{R})$ 
that sends $G_1$ to $P_{x_1}$, $G_2$ to $P_{x_2}$ and $G_3$ to $P_{x_3}$. By linearity of $\phi$ 
(in particular, $\phi$ preserves alignment), we have $\phi(S_V^{G_1}) = S_{\phi(V)}^{P_{x_1}}$ 
and similar equalities for $\phi(S_V^{G_2})$ and $\phi(S_V^{G_3})$. Also, $\phi(V)$ 
is a rational variety parameterized by $\phi \circ p$.

Thus, if $\mathcal{C} \subset S_V^{G_3}$, then $\phi(\mathcal{C}) \subset S_{\phi(V)}^{P_z}$. 
By this argument, we only have to prove that there is no rational space curve $\phi(V)$ 
for which $S_{\phi(V)}^{P_{x_1}} \cap S_{\phi(V)}^{P_{x_2}} \cap S_{\phi(V)}^{P_{x_3}}$ contains 
another curve different than $\phi(V)$.
For convenience, $\phi(V)$ is denoted by $V$ and $\phi(\mathcal{C})$
is denoted by $\mathcal{C}$ in what follows. 

\begin{remark}\label{Rhypo1}
Since $V$ does not lie in the plane
at infinity, we can switch to the affine setting. 
\end{remark}

We now have ``nicer" expressions for our surfaces:
\begin{itemize}
\item $S_V^{P_{x_1}} = \left\lbrace (x_1,p_2(t_1),p_3(t_1)) \mid x_1, t_1 \in \mathbb{R} \right\rbrace$,
\item $S_V^{P_{x_2}} = \left\lbrace (p_1(t_1),x_2,p_3(t_1)) \mid x_2, t_1 \in \mathbb{R} \right\rbrace$,
\item $S_V^{P_{x_3}} = \left\lbrace (p_1(t_1),p_2(t_1),x_3) \mid x_3, t_1 \in \mathbb{R} \right\rbrace$.
\end{itemize}
Since $\mathcal{C}\subset S_V^{P_{x_1}}$, there is a (not necessarily rational) parameterization 
of $\mathcal{C}$ given by $q: t_1 \in \mathbb{R} \mapsto (q_1(t_1), p_2(\varphi(t_1)), p_3(\varphi(t_1)))$,
where $q_1$ and $\varphi$ are continuous piecewise smooth maps.

\begin{remark}\label{Rhypo2}
We see that $\varphi$ (resp.\ $q_1$) is not locally constant:
otherwise, a part of $\mathcal{C}$ would be
included in a straight line (resp.\ a plane), which contradicts the fact
that $V$ is not planar. 
\end{remark}

We can thus pick a small interval $I := ( a, b ) \subset \mathbb{R}$
such that $q_{|I}$ is injective and so, without loss of generality, we assume that $\varphi_{|I}(t_1) = t_1$.

Also, since $V$ is not planar and $p$ is rational, the sets of non regular points of 
the three maps $\pi_{x_1} : t_1 \mapsto (p_2(t_1), p_3(t_1))$, $\pi_{x_2} : t_1 \mapsto (p_1(t_1), p_3(t_1))$ 
and $\pi_{x_3} : t_1 \mapsto (p_1(t_1), p_2(t_1))$ are finite. Shrinking $I$ if necessary,
we assume there is no such non regular point in $I$.

Now, we have a local curve $q(I) = \left\lbrace (q_1(t_1), p_2(t_1), p_3(t_1)) \mid t_1 \in I \right\rbrace$ 
included in ${(S_V^{P_{x_1}} \cap S_V^{P_{x_2}} \cap S_V^{P_{x_3}})}$. 
Using the fact it lies on $S_V^{P_{x_2}}$, 
we have another parameterization of $q(I)$ given by $r(t_1)=$ $(r_1(t_1)$, $p_2(t_1)$, $r_3(t_1))$ 
with $r_i$ injective on $I$ and $r_i(I) \subset p_i(\mathbb{R})$. Similarly, $q(I) \subset S_V^{P_{x_3}}$ 
gives a third parameterization $s(t) =$ $(s_1(t_1)$, $s_2(t_1)$, $p_3(t_1))$ with $s_i$ injective on $I$ 
and $s_i(I) \subset p_i(\mathbb{R})$.

Comparing $s$ with $q$ and $r$, we have $s(t_1) = (s_1(t_1), p_2(t_1), p_3(t_1))$.
Lastly, since $\pi_{x_1}$ is regular on $I$, there is only one branch in $\pi_{x_1}(s(I))$ 
and so $s_1^{-1}(s(I)) = p_1^{-1}(s(I))$. The curve $\mathcal{C}$ is thus locally contained in $V$; 
it follows that $\mathcal{C} = V$.

\emph{Case (2)}. 
Let us look more closely at Remarks~\ref{Rhypo1} 
and~\ref{Rhypo2} where we used the non-planarity hypothesis:
\begin{enumerate}
\item At Remark~\ref{Rhypo1} 
we used the fact that $V$ is not contained in the plane containing $G_1, G_2$ and $G_3$.
\item At Remark~\ref{Rhypo2} we used the fact that the points $G_i$ do not belong to a plane containing $V$, $i=1,\dots, 3$.
\end{enumerate}
These facts are still valid under the assumption of $G_i \not\in \mathcal{P}$, hence the proof of \emph{Case~(1)} holds.
\end{proof}


\subsection{The general case} 
\label{SSgeneraldim}

Let us generalize the construction above.
Let $V\subset \PP^n$ be a $d$-dimensional variety parameterized as 
$$
x_i=f_i(t),\, i=0,\ldots,n,\, t=(t_0:\dots :t_{d}),
$$ 
and $G$ be a set of $n-d-1$ sufficiently generic points $G_i$ not in $V$.
Let $L$ be the $(n-d-1)$-dimensional
linear subspace  defined by a symbolic point
$\xi=(\xi_0:\dots :\xi_n)$ and the points in $G$.
Chose $d+1$ sets of $d$ random points $P_i=\{P_{i1},\ldots,P_{id}\}$, $P_{ij}\notin V$.
Let $H_i,\, i=0,\ldots,d$ be the hyperplane defined as the span of the points $\xi, G, P_i$.
Substitute  $x_i=f_i(t)$ in each $H_i$ to obtain
\begin{align}\label{Eq:hyperplanessubs}
 H_0(f_0(t):\ldots :f_{n}(t))=0 , \;\,
\dots, \;\, 
 H_d(f_0(t):\ldots :f_{n}(t))=0 .
\end{align}

The resultant of the polynomials in~\eqref{Eq:hyperplanessubs}
eliminates $t$ and returns
a polynomial $R^G_V$ in $\xi$ which vanishes on $V$ (but not only on $V$),
thus offering a necessary but not sufficient condition, see
Algorithm~\ref{Achow}.
To achieve the hypothesis of Prop.~\ref{Pgkz},
one may iterate for a few distinct pointsets $G$ thus obtaining 
implicit hypersurfaces $R_V^G=0$ whose intersection is $V$.

To compute the resultant one option is to use
interpolation in conjunction with information on its support.
This might be obtained from degree bounds on $R_V^G$, as explained below,
or by the computation of the monomials
of the polynomials in~\eqref{Eq:hyperplanessubs}
using software \texttt{ResPol} from~\cite{EmFiKoPeJ}.
The latter is analogous to the basic approach for defining the
interpolation matrix by the support obtained from the resultant polytope, see Sec.~\ref{Smatrixmethod}.  

In the general setting we do not have a result similar to Lem.~\ref{Lsurfacedegree}. 
Our method gives rise to extraneous factors, whose
expressions and powers are less clear than in the case of space curves. 
From experiments, it appears that there are only extraneous factors of degree $\codim(V)$.
Of course, there are other extraneous factors depending on the chosen method for computing the resultant,
e.g., interpolation, or as the determinant of the Macaulay matrix.

To find the degree of $R^G_V$ in $\xi$, first note that the degree in $t$ of
every $H_i(f_0(t):\ldots:f_{n}(t))$
is $\delta=\max\{\deg_t f_j(t), \, j=0,\ldots,n\}$,
and that the coefficients of the $H_i$'s are linear polynomials in $\xi$.
The resultant of these polynomials has degree in the coefficients of each $H_i$ equal to
$\delta^d$, therefore total degree $(d+1)\delta^d$, see e.g.~\cite[Thm.3.1]{CLO2}.

A tighter bound on the degree of $R_V^G$ can be obtained by considering sparse resultants
and {\em mixed volumes} \cite[Ch.7]{CLO2}. Denote
\[
MV_{-i}=MV(H_0,\ldots,H_{i-1},H_{i+1},\ldots,H_d) , \; 0\le i\le d,
\]
the mixed volume of all polynomials excluding $H_i$.
The degree of the sparse resultant in the coefficients of $H_i$ is known
to equal $MV_{-i}$, therefore its total degree equals $\sum_{i=0}^d MV_{-i}$.
Both degree bounds above are typically much larger than the degree of $V$.

An issue arises of course at sampling: all generated points $\xi$ lie on
$V$, whereas we are trying to compute a hypersurface containing $V$.
We expect that the kernel of $M$ will have large dimension and among
the kernel vectors we may choose one or more ``small" vectors to define
the implicit equation. Here  ``small'' may refer to the number of non-zero vector entries,
or to the total degree of the monomials corresponding to its non-zero entries.

We can avoid interpolation if there exists a rational or, even better,
determinantal formula, for the resultant,
such as Sylvester's for curves in $\PP^n$,
or D'Andrea's formula for sparse resultants~\cite{DAnd02},
which are optimal for generic coefficients.
For arbitrary coefficients, an infinitesimal perturbation may be applied.

\begin{algorithm}[ht]
\caption{Implicit equation}\label{Achow} 
\DontPrintSemicolon \BlankLine
\Input{$V\subset\PP^d$, parameterized by $x_i=f_i(t), i=0,\dots,n$.}
\Output{A polynomial vanishing on $V$.}
(1)\, Define $(n-d-1)$-dimensional linear subspace $L$ by 
affinely independent random points
$G= \{G_1,\dots,G_{n-d-1}\}$ : $G_i\not\in V$, and
consider symbolic point $\xi=(\xi_0,\dots,\xi_n)$. 
\\ (2)\, Define $d+1$ hyperplanes $H_i$ through $G,\xi,P_i$, for  
random points $P_i$ affinely independent from points in $G$.
\\ (3)\, Set $x_i=f_i(t)$ in the $H_i$: their
resultant, where we eliminate $t$ is the sought polynomial in $\xi$.
Compute it by formula or by interpolation
as in Sec.~\ref{Smatrixmethod}.
\end{algorithm}


\subsection{Examples and complexity}\label{SSefficiency}

In this section we illustrate our method through examples of curves in 3D and 4D space and 
a surface in 4D space.

We switch from the projective to the affine setting and
set $\xi=(\xi_1,\ldots,\xi_n)\equiv(x_1,\ldots,x_n)$, for emphasizing these
are the implicit variables. 
For curves in any ambient dimension
we compute the resultant directly using the Sylvester matrix,
and also by interpolation, employing degree bounds relying on mixed volume.
For $d>1$ we use the Macaulay (or the toric resultant) matrix.
Interpolating the resultant leads to matrices with very large kernels:
on the upside, the polynomials we obtain from the kernel vectors
are of degree no greater than those of the former method.
Moreover, among them we can find a number of polynomials of small degree
and, often, smaller than the degree predicted by degree bounds:
these polynomials define the variety set-theoretically, 
in other words, as the intersection of a few surfaces.

\begin{example}  \label{Exam:twistedcubic}
Consider the twisted cubic curve {\em affinely} parameterized as:
\[
V = \{(t , t^2 , t^3 ) \in \RR^3 \, :\, t \in \RR \}.
\]
The implicit equations of $V$ are 
\begin{equation}\label{Eq:twistcubic}
x_1^2 - x_2 = x_2^2 - x_1 x_3 = x_1 x_2 - x_3 = 0 .
\end{equation}

Let $L$ be the line passing through symbolic point $\xi=(x_1,x_2,x_3)$,
and generic point $G\not\in V$ (here the set $G$ consists of only one point also denoted $G$).
We define two random planes
$H_1(x_1 , x_2 , x_3 )$,  $H_2(x_1 , x_2 , x_3 )$,
intersecting at $L$ by considering additional random points $P_1, P_2\notin L$, respectively.
Then, the Sylvester resultant of $H_1(t , t^2 , t^3 )=H_2(t , t^2 , t^3 )=0$
is a polynomial of degree~6 in $\xi$ which factors into
the degree~3 polynomial 
\begin{eqnarray*}
32x_{2}-16x_{3}+56x_{1}x_{3}+16x_{1}x_{2}-80x_{2}^{2}-
\\
-32x_{1}^{2} - 40x_{3}^{2} + 42x_{2}^{3} - 2x_{{3}}x_{{2}}x_{{1}}+56x_{{3}}x_{{2}}-
\\
-5x_{{3}}{x_{{2}}}^{2}+
5{x_{{ 3}}}^{2}x_{{1}}-8x_{{3}}{x_{{1}}}^{2}-48{x_{{2}}}^{2}x_{{1}}+24x
_{{2}}{x_{{1}}}^{2}
\end{eqnarray*}
and the expected extraneous linear factor raised to the power 3.

This yields a surface containing $V$ but not of minimal degree.
Repeating the procedure 3 times, the ideal of the resulting polynomials equals the ideal defined
from the polynomials in \eqref{Eq:twistcubic}.

Alternatively, we may interpolate the Sylvester resultant above.
We take as predicted support the lattice points
in a 3-simplex of size 6. The $84\times84$ matrix constructed has a kernel of dimension 65.
The corresponding 65 kernel polynomials are of degrees from 2 to 6. 
Among them, the three polynomials of
minimal degree 2: $x_1 x_3 - x_2^2, ~x_1 x_2 - x_3, ~ -x_2 + x_1^2$,
define the twisted cubic exactly.

In contrast, \cite[Sec.~3.3]{DalStu95} gives 16 (homogeneous) implicit equations, all of degree 3.
\end{example}


\begin{example}  \label{Exam:2cylinders}
Consider the space curve in Fig.~\ref{F2cylinders} affinely parameterized as:
\[
V = \left\{
\left( \frac{1-t^2}{1+t^2} , \frac{2t}{1+t^2} , \left(\frac{1-t^2}{1+t^2}\right)^2 \right)
\in \RR^3 \, :\, t \in \RR \right\}.
\]
It is the intersection of two cylinders:
\begin{equation}\label{Eq:2cylinders}
x_1^2-x_3 = x_2^2+x_3-1 = 0 .
\end{equation}

Let line $L$ be defined from the symbolic point $\xi=(x_1,x_2,x_3)$,
and ``generic" point $G\not\in V$.
Define two random planes $H_1$ and $H_2$ that intersect at $L$,
by choosing random points $P_1, P_2\notin L$, respectively.
Then, the Sylvester resultant of
$$
H_1(\frac{1-t^2}{1+t^2} , \frac{2t}{1+t^2} , (\frac{1-t^2}{1+t^2})^2),\;\;
H_2(\frac{1-t^2}{1+t^2} , \frac{2t}{1+t^2} , (\frac{1-t^2}{1+t^2})^2)
$$
is a polynomial of degree~8 in $\xi$ which factors into the following degree~4 polynomial:
\begin{eqnarray*}
29+120x_1-56x_2-182x_3+168x_1x_2-110x_1^2-78x_2^2+\\
+90x_2^2x_3-12x_2x_3+70x_1^2x_3-18x_1x_3 -120x_1x_2^2+40x_1^2x_2+\\
+156x_3^2+56x_2^3+49x_2^4+25x_1^4-48x_3^2x_2-72x_3^2x_1+\\
+12x_3^2x_2^2+28x_3x_2^3+12x_3^2x_1^2-168x_2^3x_1+222x_2^2x_1^2-\\
-30x_3x_1^3-120x_2x_1^3-48x_3x_2x_1-102x_3x_2^2x_1+76x_3x_1^2x_2
\end{eqnarray*}
and the expected extraneous linear factor raised to the power~4.
This yields a surface containing $V$ but not of minimal degree.
As in the previous example,
repeating the procedure 3 times,
the ideal of the resulting polynomials equals the ideal defined
from the polynomials in \eqref{Eq:2cylinders}. 
Note that all iterations produce polynomials of degree 8
which all factor into a degree~4 polynomial vanishing on $V$ and a degree~1 polynomial 
raised to the power~4, as predicted by Lem.~\ref{Lsurfacedegree}.

Alternatively, we interpolate the Sylvester resultant above using as support the lattice points
in a 3-simplex of size 8. The $165\times165$ matrix constructed has a kernel of dimension 133.
The degrees of the corresponding 133 kernel polynomials vary from 2 to 8. 
Among them, there are two polynomials of
degree 2. They coincide with those in \eqref{Eq:2cylinders}.

The equations above were obtained by choosing random points with integral coordinates and relatively close to the curve.
When we increase the range or allow for non-integral points, in order to have better 
chances to avoid the non-generic points,
the size of the coefficients increase.
For example, using random points with integral coordinates in a box of size 100 around 
the curve yields equations as below, where we omit most terms:
{\small \begin{eqnarray*}
3983438998535755975578507593x^4 + \dots 
- 6421697880560981054975490304
\end{eqnarray*}}
\end{example}


\begin{example}\label{ExSendra-Viviani}
We illustrate Lem.~\ref{Lsurfacedegree} in two more examples.
For the space curve of parametric degree 4 in~\cite{RueSenSen}, 
we compute the 3 implicit equations of degree 4 in  0.171 sec.
For the Viviani curve which has parametric degree 4  
and implicit equations: $x_1^2+x_2^2+x_3^2=4\,a^2$,\; $(x_1-a)^2+x_2^2=a^2$,
we compute 3 implicit equations of total implicit degree~4, 
and degree 74 in the parameter $a$, in 7.72~sec.
\end{example}

\begin{example}\label{Ex4Dcurvemain}
Consider the curve in $\mathbb{R}^4$ with parameterization:
\[
(x_1,x_2,x_3,x_4) = (t^2-t-1,\, t^3+2\,t^2-t,\, t^2+t-1,\, t^3-2\,t+3)
\]
The curve is defined by 5 implicit equations of degrees 1,2,2,2 and 3.
Our algorithm computes 5 equations of degree 6 in 0.06 sec. 
These equations contain linear extraneous factors raised to the power 3. 
When these are divided out we obtain 5 degree 3 equations which define the curve set theoretically. 
\end{example}

\begin{example}
We tested our method on a surface in $\mathbb{R}^4$ with parametric degree~2.
It has an implicit representation defined by 7 equations all of degree 3.
The resultant of the hyperplane equations is of degree 12, while the polynomial 
computed as the determinant of the Macaulay matrix is of degree 15 and
factors into 4 polynomials:
an irreducible polynomial of degree 4 that contains the surface,
two polynomials due to the Macaulay resultant, and a power of a linear polynomial
due to our method. A total of 5 polynomials define the surface set-theoretically.
\end{example}

In all examples in this section, we used 
our Maple implementation of Algorithm~\ref{Achow}. 
We compared to Maple's native command 
\texttt{algcurves[implicitize]} on a laptop with a 2GHz Intel Celeron processor running Maple~18.
Our implementation is slower per output equation for curves of degree larger than~6, see Fig.~\ref{F3Efficiency}.
For curves of moderate and high degrees, the Sylvester resultant computation (or interpolation) 
represents more than 99\% of our algorithm's computation time.

\begin{figure}[ht]
\!\includegraphics[width=0.4\textwidth]{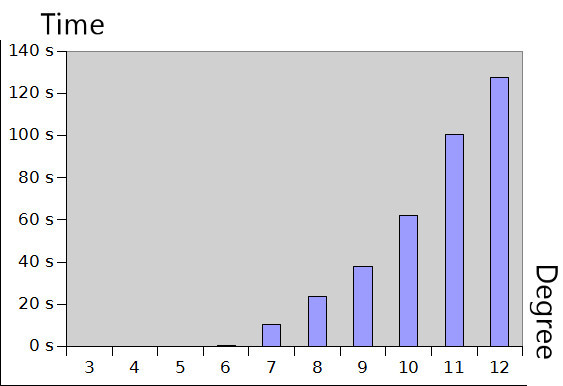}\qquad
\includegraphics[width=0.4\textwidth]{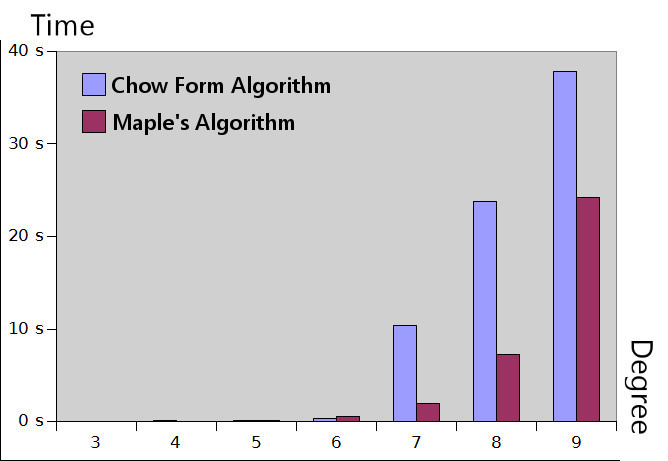}
\centering\caption{Runtime (sec) per output equation vs parametric degree.
Our method's times are in (light) blue on both graphs, Maple's native command's are in (dark) red.
\label{F3Efficiency}}
\end{figure}

Though the runtime per equation is lower, the number of implicit equations that
\texttt{implicitize} outputs cannot be predicted and it seems to be only dependent on the input degree:
for inputs of the same degree it always returns the same number of implicit equations.
For degrees larger than 8, their number is greater than 100.
In contrast, our algorithm can return the implicit equations one by one, and
Thm.~\ref{Tthreesurf} guarantees that~3 of them 
define the curve set-theoretically.
For parameterizations of degree up to~4, \texttt{implicitize} computes a small number of 
implicit equations whose total degree
is lower than the degree of those computed by our method. 
On the other hand, our algorithm always outputs equations of the same total degree equal 
to the degree of the parameterization, by Lem.~\ref{Lsurfacedegree}.
For instance, when the degree of the parameterization is 4,
\texttt{implicitize} typically returns~7 equations
of degree~3 and~11 equations of degree~4,
whereas our algorithm returns~3 equations of degree~4.
For Examp.~\ref{Exam:2cylinders}, \texttt{implicitize} computes the two equations in \eqref{Eq:2cylinders}.
This cannot be done by our algorithm unless we know which apex $G$ for the conical
surface should be used: for this example these are points at infinity.

The largest example that we computed is of a parametric space curve of degree 18 in less than 3 min.

We also compared our method against implicitization using Groebner basis and the method of 
matrix representations based on syzygies \cite{BusLBa12} using the authors' Maple code.
For parameterizations of degree up to 6 all methods are comparable in running time.
For larger degrees Groebner basis outperform our method by a factor that depends on the degree and the number of monomials 
in the parametric polynomials and ranges from 10 to $10^3$. 
The method based on syzygies yields a matrix $M$ which has the drop of rank property and, moreover,
its minors of size $\Rank(M)$ give a system of implicit equations of the variety. However, computing these
minors, unlike the computation of $M$, is very inefficient and it is not in the spirit of this method. 
On the upside, our algorithm always outputs 3 implicit equations, while in both other methods their number depends on the input 
and is typically much larger than 3.

Our algorithm was also implemented in Sage with similar runtimes as in Maple.


\section{Future work} \label{Sfuture}

Interesting mathematical questions arise in studying the complexity
of our method, namely how to bound the degree and the number of computed conical implicit surfaces
in the general case by results similar to Thm.~\ref{Tthreesurf}. This requires to study the extraneous
factors of the computed equations and to distinguish them into those coming from the chosen method to compute 
the resultant, and those inherited in our method.

Interpolation matrices
extend to other geometric operations such as surface-surface
intersection, assuming one surface is either
parameterized or given by a point cloud, and the second is parameterized.
The first surface is captured by matrix $M'$ through a point
sample. The parameterization of the
second surface is substituted into the last row $S(x)$.
Now $\det M$ is an implicit equation of the intersection curve,
defined in the parametric plane of the second surface.
We may avoid developing the determinant
but instead we may operate on the matrix.
For instance, we can sample the intersection curve by fixing values of
one parameter and solving the resulting univariate polynomial
$\det [ M' | S(t) ]$ in the other parameter $t$, much
like we did with ray shooting.

The interpolation matrices seem to extend beyond the codimension 1 case:
given a variety of codimension  $>1$ 
(i.e., in~\eqref{Eaff_parametrization} we have $t=(t_1,\ldots,t_d)$ with $d \le n-2$), 
we can construct a matrix $M(x)$ as in Sect.\ref{Smatrixmethod}
that still ``represents'' the variety, under some assumptions on the support $S$ used:
point $q\in \mathbb{C}^n$ lies in the variety iff $\Rank(M(q)) = \Rank(M')$.
However, in that situation, matrix $M'$ is never of full rank.
The implicit equations of the variety can be found among the non-zero 
minors of $M(x)$ of size $\Rank(M')$, but this computation is very expensive. 
Similar properties hold for the matrix representations in~\cite{Buse14}
but these matrices are smaller in size.

%
%
\section*{Acknowledgments}
IE, CK, CL belong to team AROMATH, joint between
INRIA Sophia-Antipolis and NKUA; their
work is partially supported as part of Project 
ARCADES that has received funding from the European Union's Horizon 2020 
research and innovation programme under the Marie Sklodowska-Curie grant agreement No 675789.
We thank L.Bus\'e and J.R.Sendra for discussions.

\end{document}